\documentclass[a4paper,12pt]{amsart}

\usepackage[T1]{fontenc}
\usepackage[utf8]{inputenc}
\usepackage{lmodern}
\usepackage[english]{babel}

\usepackage{subcaption}
\usepackage{tikz}
\usetikzlibrary{hobby}
\usetikzlibrary{patterns}

\usepackage[
	left=2.5cm,
	right=2.5cm,
	top=3.5cm,
	bottom=3.5cm,
	heightrounded,
	bindingoffset=0mm
]{geometry}

\usepackage{amsmath,amssymb,mathrsfs}
\usepackage{mathtools}

\usepackage{subcaption}

\usepackage{bookmark}
\usepackage{hyperref}

\usepackage{color}
\usepackage{tikz}
\usetikzlibrary{patterns}
\usetikzlibrary{intersections,decorations.markings}

\makeatletter
\tikzset{
  use path for main/.code={%
    \tikz@addmode{%
      \expandafter\pgfsyssoftpath@setcurrentpath\csname tikz@intersect@path@name@#1\endcsname
    }%
  },
  use path for actions/.code={%
    \expandafter\def\expandafter\tikz@preactions\expandafter{\tikz@preactions\expandafter\let\expandafter\tikz@actions@path\csname tikz@intersect@path@name@#1\endcsname}%
  },
  use path/.style={%
    use path for main=#1,
    use path for actions=#1,
  }
}
\makeatother

\newtheorem{thm}{Theorem}[section]
\newtheorem{coro}[thm]{Corollary}
\newtheorem{lemma}[thm]{Lemma}
\newtheorem{prop}[thm]{Proposition}

\theoremstyle{definition}
\newtheorem{defn}[thm]{Definition}
\newtheorem{remark}[thm]{Remark}

\newcommand{\R}{\mathbb R}
\newcommand{\B}{\mathbb B}

\newcommand{\eps}{\varepsilon}
\newcommand{\di}{\mathop{}\!\mathrm{d}}

\newcommand{\dE}{\partial E}

\let\epsilon\varepsilon

\DeclareMathOperator{\dist}{dist}

\mathchardef\ordinarycolon\mathcode`\:
\mathcode`\:=\string"8000
\begingroup \catcode`\:=\active
  \gdef:{\mathrel{\mathop\ordinarycolon}}
\endgroup

\usepackage[initials]{amsrefs}

\numberwithin{equation}{section}

\usepackage{url}

\DefineSimpleKey{bib}{nomecorto}

\newcommand\link[1]{\url{#1}}

\BibSpec{webpage}{%
  +{}{\PrintAuthors} {author}
  +{,}{ \textit} {title}
  +{,}{ } {date}
  +{,}{ } {note}
  +{,}{ \link} {link}
  +{.}{ } {transition}
}

\begin{document}

\title[]{On critical points of the relative fractional perimeter}

%11-01-2018

\author{Andrea Malchiodi}
\author{Matteo Novaga}
\author{Dayana Pagliardini}

\address{Andrea Malchiodi \newline
Scuola Normale Superiore,
Piazza dei Cavalieri 7,
56126 Pisa, Italy}
\email{andrea.malchiodi@sns.it}

\address{Matteo  Novaga \newline
Dipartimento di Matematica, Università di Pisa, Largo B. Pontecorvo 5, 56217 Pisa, Italy}
\email{matteo.novaga@unipi.it}

\address{Dayana Pagliardini \newline
Scuola Normale Superiore,
Piazza dei Cavalieri 7,
56126 Pisa, Italy}
\email{dayana.pagliardini@sns.it}

\begin{abstract}
We study the localization of sets with constant nonlocal mean curvature and  prescribed small volume in a bounded open set with smooth boundary, proving that they are {\em sufficiently close} to critical points of a suitable non-local potential.  We then consider the fractional perimeter in  half-spaces. We prove the existence of a minimizer under fixed volume constraint, showing some of its properties such as smoothness and symmetry, being a graph in the $x_N$-direction, and characterizing its intersection with the hyperplane $\{x_N=0\}$.
\end{abstract}

\date{\today}

\subjclass[]{}

\maketitle

\begin{center}
\end{center}

\tableofcontents

\section{Introduction}
Isoperimetric problems play a crucial role in several areas such as geometry, linear and nonlinear PDEs, probability, Banach space theory 
and others. 
Its classical version consists in studying least-area sets contained in a fixed region (the Euclidean space or any given domain). If 
the ambient space is an  $N$-dimensional manifold $M^N$ with or without boundary, the goal would be to find, among all the compact hypersurfaces $\Sigma \subset M$ which bound a region $\Omega$ of given volume $V(\Omega)=~m$ (for $0<m<V(M)$), those of minimal area $A(\Sigma)$. 
Such a  region $\Omega$ is called an {\em isoperimetric region} and its boundary $\Sigma$ is called an {\em isoperimetric hypersurface}.

A first general existence and regularity result can be obtained for example combining the results in \cite{A} with those in \cites{G,GMT}. In particular we have that if $N\le 7$, $\Sigma$ is smooth. 
We also refer the reader to the interesting survey \cite{R}.

Beyond the existence and the regularity problem, it is also interesting to  study the geometry and the topology of the solutions, and to give a qualitative description of the isoperimetric regions. Concerning these aspects,  we recall that in \cite{MJ} it was proved that a region of small prescribed volume in a smooth and compact Riemannian manifold has asymptotically (as the volume tends to zero) at least as much perimeter as a round ball.

Afterwards, regarding critical points of the perimeter relative to a given set, in \cite{Fall} the existence of surfaces with the 
shape of half spheres was shown, surrounding a small volume near nondegenerate critical points of the mean curvature of the  boundary of an open smooth set in $\mathbb{R}^3$. It was  proved that the boundary mean curvature determines the main terms, studying the problem via a Lyapunov-Schmidt reduction.
In  \cite{F}, the same author showed that isoperimetric regions with small volume in a bounded smooth domain $\Omega$ are near global maxima of the mean curvature of $\Omega$.

Results of this type were proven in \cite{D} and \cite{Y}. The authors considered closed manifolds and proved that isoperimetric regions with small volume locate near the maxima of the scalar curvature. In \cite{Y}  a viceversa was also shown: for every non-degenerate critical point $p$ of the scalar curvature there exists a neighborhood of $p$ foliated by constant mean curvature hypersurfaces. Moreover, in \cite{T}  the boundary regularity question for the capillarity problem was studied.

In  recent years fractional operators have received considerable attention for both in pure and applied motivations. 
In particular, regarding perimeter questions, in \cite{ADPM}  the link between the fractional perimeter and the classical De Giorgi's perimeter was analyzed, showing the equi-coercivity and the $\Gamma$-convergence of the fractional $s$-perimeter, up to a scaling factor depending on $s$, to the classical perimeter in the sense of De Giorgi and a local convergence result for minimizers 
was deduced.

Another relevant result about fractional perimeter was obtained in \cite{FFMMM}, generalizing a quantitative isoperimetric inequality to the fractional setting. Indeed, in the Euclidean space, it is known that among all sets of prescribed measure, balls have the least perimeter, i.e.\ for any Borel set $E \subset \mathbb{R}^N$  of finite Lebesgue measure, one has 
\begin{equation}\label{isopclass}
N|B_1|^{\frac{1}{N}}|E|^{\frac{N-1}{N}}\le P(E)
\end{equation}
with $B_1$ denoting the unit ball of $\mathbb{R}^N$ with center at the origin and $P(E)$ is the distributional perimeter of $E$. The equality in $\eqref{isopclass}$ holds if and only if $E$ is a ball. 

In \cite{FMM}   a similar result for the fractional perimeter $P_s$ (defined as in $\eqref{PersE}$) was obtained, improved then in \cite{FFMMM}  showing the following  
fact: 
%\begin{thm}[\cite{FFMMM}*{Theorem $1.1$}]
for every $N\ge 2$ and any $s_0\in (0,1)$  there exists $C(N,s_0)>~0$ such that
\begin{equation}\label{fractisoperim}
P_s(E)\ge \frac{P_s(B_1)}{|B_1|^{\frac{N-s}{N}}}|E|^{\frac{N-s}{N}}\Biggl \{1+ \frac{A(E)^2}{C(N,s)} \Biggr \}
\end{equation}
whenever $s\in [s_0,1]$ and $0<|E|<\infty$.
%\end{thm}
Here 
\[
A(E):=\inf \Biggl \{ \frac{|E \triangle (B_{r_E}(x))|}{|E|}: x \in \mathbb{R}^N \Biggr \}
\]
stands for the \emph{Fraenkel asymmetry} of $E$, measuring the $L^1$-distance of $E$ from the set of balls of volume $|E|$ and $r_E=(|E|/|B_1|)^{1/N}$ so that $|E|=|B_{r_E}|$.

In the same spirit of extension of classical results to the fractional setting, we also mention \cite{MV}. Here the authors modify the classical Gauss free energy functional used in capillarity theory by considering surface tension energies of nonlocal type. They analyzed a family of problems including a nonlocal isoperimetric problem of geometric interest. In particular, given $N\ge 2$, $s\in (0,1)$, $\lambda \ge 1$ and $\eps \in [0,\infty]$ they considered the family of interaction kernels $\bold{K}(N,s,\lambda, \eps)$, i.e.\ even functions $K:\mathbb{R}^N \setminus \{0 \} \rightarrow [0,+\infty)$ such that
\[
\frac{\chi_{B_\eps}(z)}{\lambda |z|^{N+s}}\le K(z)\le \frac{\lambda}{|z|^{N+s}}\quad \forall \; z\in \mathbb{R}^N \setminus \{0 \}
\]
where $B_\eps(x)$ is the ball of center $x$ and radius $\eps$. Taking $\Omega \subset \mathbb{R}^N$ and $\sigma \in (-1,1)$ the authors studied the nonlocal capillarity energy of $E \subset \Omega$ defined as
\[
\mathcal{E}(E)=\int_E\int_{E^C \cap \Omega}K(x,y)\di x \di y+\sigma \int_E \int_{\Omega^C}K(x,y)\di x \di y
\]
with $K\in \bold{K}(N,s,\lambda, \eps)$, giving existence and regularity results, density estimates and new equilibrium conditions with respect to those of the classical Gauss free energy.

As it concerns constant nonlocal mean curvature, we mention the paper \cite{CFW16}, where it was proved the existence of Delaunay type surfaces, i.e.\ a smooth branch of periodic topological cylinders with the same constant nonlocal mean curvature. We also refer to \cite{Mi}, where the author constructs two families of hypersurfaces with constant nonlocal mean curvature.

Moreover we notice that recently, in \cite{M},  the axial symmetry of {\it smooth} critical points of the fractional perimeter in a half-space was shown, using a variant of the moving plane method.

Motivated by these results, in the first part of this paper our aim is to study the localization of sets with constant nonlocal mean curvature and small prescribed volume relative to an open bounded domain. The notions of relative fractional perimeter $P_S(E,\Omega)$ and of relative fractional mean curvature $H_s^\Omega$ we are going to use are given by formulas \eqref{PersE} and \eqref{NMC} in the next section.

\begin{thm}\label{mainth}
Let $\Omega \subseteq \mathbb{R}^N$ be a bounded open set with smooth boundary and $s\in (0,1/2)$.

For $x$ in a given compact set $\Theta$ of $\Omega$, set 
\[
V_\Omega(x):=\int_{\Omega^C}\dfrac{1}{|x-y|^{N+2s}}\di y. 
\]
Then for every strict local extremal or non-degenerate critical point $x_0$ of $V_\Omega$ in $\Omega$, there exists 
$\overline{\eps} > 0$ such that for every  $0 < \eps < \overline{\eps}$ there exist spherical-shaped surfaces 
with constant $H_s^\Omega$ curvature and enclosing volume identically equal to $\eps$,  approaching $x_0$ as $\eps \to 0$. 
\end{thm}

Notice that in  \eqref{PersE} (as well as in the above formula) we are using the exponent $2s$ in the denominator, and hence 
in our notation the range $(0,1/2)$ for $s$ is  natural. 
One of the main tools for proving this result relies on  the non-degeneracy of spheres with respect to 
the linearized non-local mean curvature equation, which follows from a result in \cite{CFW}. 
 After non-degeneracy is established, we can use a 
 Lyapunov-Schmidt reduction to study a finite-dimensional problem, which is treated by 
carefully expanding  the relative  fractional perimeter of balls with small volume. Thanks to classical results 
in min-max theory, we obtain as a corollary a multiplicity result. Here and in the following, cat$(\Omega)$ denotes the Lusternik-Schnirelman category of the set $\Omega$ (see \cite{J} and Section \ref{sec2} below for more details).
\begin{coro}\label{Maincoro}
Let $\Omega \subseteq \mathbb{R}^N$ be a bounded open set with smooth boundary. 
Then there exists 
$\overline{\eps} > 0$ such that for every  $0 < \eps < \overline{\eps}$  there exist at least cat$(\Omega)$ spherical-shaped surfaces 
with constant $H_s^\Omega$ curvature and enclosing volume identically equal to $\eps$. 
\end{coro}

\

In the last part of this work we aim to study the existence and some properties of sets minimizing the fractional perimeter in a particular domain, namely a half-space:
\begin{thm}\label{Mainthm2}
There exists a minimizer $E$ for the problem 
\begin{equation}\label{Phalf}
\inf \Big \{ P_s(A, \mathbb{R}_+^N),\; |A|=m \Big \},\quad m\in (0,+\infty),
\end{equation}
where $\mathbb{R}_+^N:=\{ x \in \mathbb{R}^N:x_N>0\}$. Moreover $\partial E$ is a radially-decreasing symmetric graph of class $C^\infty$ in the interior, intersecting orthogonally the hyperplane $\{x_N=0\}$. 
\end{thm}

This result is proved by showing first the existence of a properly rearranged minimizing 
sequence which is axially symmetric and graphical over the boundary hyperplane. After 
this is done, we employ some results from \cite{BFV}, \cite{CRS}, \cite{MV} to 
prove a diameter bound and smoothness of the minimizing limit. 

\

The paper is organized as follows: In Section \ref{sec2} we introduce some notation on fractional perimeter and mean curvature, and we show some preliminary results, especially on the  linearized fractional mean curvature. 
We prove in 
particular the minimal degeneracy for spheres, also relative to suitably large domains. 
%and we study its H\"older regularity at smoothly perturbed spheres. 
In Section \ref{sec3} we prove Theorem \ref{mainth}  via a Lyapunov-Schmidt reduction and Corollary \ref{Maincoro} through a well known result about the Lusternik-Schnirelman category. Finally, in Section~\ref{sec4} we prove Theorem \ref{Mainthm2} in two steps: the existence of minimizers in a bounded domain is a rather standard consequence of the direct method of Calculus of Variations.  We then show the symmetry of minimizers and, using the density estimates holding for the fractional perimeter, we prove also the connectedness and hence the free minimality.

\

\begin{center}
{\bf Acknowledgements} 
\end{center}

A.M. has been supported by the project {\em Geometric Variational Problems} from Scuola Normale Superiore, 
A.M. and D.P.  by MIUR Bando PRIN 2015 2015KB9WPT$_{001}$, M.N. by the University of Pisa via the grant PRA-2017-23. The authors are all members of GNAMPA as part of INdAM.

\section{Notation and preliminary results}\label{sec2}
In this section we introduce the notation that will be used throughout the paper.
We first define fractional perimeter spaces and fractional mean curvature, listing some of their properties.

%Let $s\in (0,1/2)$ and let $\Omega \subset \mathbb{R}^N$ be an open set. Following \cite{DNPV}, we let  $\mathcal{H}^s(\Omega)$ be 
%the space of functions 
%\[
%\mathcal{H}^s(\Omega):=\Big \{u\in L^2(\Omega):\frac{|u(x)-u(y)|}{|x-y|^{N/2+s}}\in L^2(\Omega \times \Omega) \Big \}.
%\]
%This is a Banach space with  natural norm
%\[
%\|u\|_{\mathcal{H}^s(\Omega)}:=\Big( \int_\Omega |u|^2\di x+\int_\Omega \int_\Omega \frac{|u(x)-u(y)|^2}{|x-y|^{N+2s}}\di x \di y\Big)^{\frac{1}{2}}.
%\]
%Here
%\[
%[u]_{\mathcal{H}^s(\Omega)}:=\int_\Omega \int_\Omega \frac{|u(x)-u(y)|^2}{|x-y|^{N+2s}}\di x \di y\Big)^{\frac{1}{2}}
%\]
%is the so-called {\em Gagliardo semi-norm of $u$}.

For $0<s<1/2$ the {\em fractional perimeter} (or $s$-perimeter) of a measurable set $E\subset \mathbb{R}^N$ is defined as
\begin{equation}\label{Pers}
P_s(E):=\int_E\int_{E^C}\dfrac{\di x\di y}{|x-y|^{N+2s}},
\end{equation}
where $E^C$ is the complement of $E$.
It has also a simple representation in terms of the usual seminorm in the fractional Sobolev space $H^s(\mathbb{R}^N)$, that is
\[
P_s(E)=[ \chi_E]_{H^s(\mathbb{R}^N)}^2:=\int_{\mathbb{R}^N}\int_{\mathbb{R}^N}\dfrac{|\chi_E(x)-\chi_E(y)|^2}{|x-y|^{N+2s}}\di x\di y,
\]
where $\chi_E$ denotes the characteristic function of $E$.
We say that a set $E\subset \mathbb{R}^N$ has {\em finite $s$-perimeter} if $\eqref{Pers}$ is finite.
If $E$ is an open set and $\partial E$ is a smooth bounded surface, we have from \cite{ADPM}*{Theorem~$2$} that as $s\rightarrow 1/2$
\begin{equation}\label{GammaconvPer}
(1-2s)P_s(E)\rightarrow \omega_{N-1} P(E),
\end{equation}
where $\omega_{N-1}$ denote the volume of the unit ball in $\mathbb{R}^{N-1}$ for $N\ge 2$ and $P(E)$ is the perimeter in the sense of De Giorgi.

This nonlocal notion of perimeter can be considered also relative to a bounded open set $\Omega$ by 
the formula 
\begin{equation}\label{PersE}
P_s(E,\Omega):=\int_{E}\int_{\Omega \setminus E}\dfrac{\di x\di y}{|x-y|^{N+2s}}.
\end{equation}

\begin{defn}
We say that a set $E\subset \mathbb{R}^N$ is a {\em minimizer} for the fractional perimeter relative to $\Omega$ if 
\begin{equation}
P_s(E,\Omega)\le P_s(F,\Omega)
\end{equation}
for any measurable set F that coincides with $E$ outside $\Omega$, i.e.\ $F \setminus \Omega = E \setminus \Omega$.
\end{defn}
Let $s\in (0,1/2)$ and let $\Omega \subseteq \mathbb{R}^N$ be an open set. We recall that the fractional mean curvature of a set $E$ at a point $x\in \partial E$ is defined as follows
\begin{equation}\label{NMC}
H_s^\Omega(\partial E)(x) :=\int_{\Omega}\frac{\chi_{E^c \cap \Omega}(y)-\chi_E(y)}{|x-y|^{N+2s}}\di y,
\end{equation}
(see \cite{MV}*{Theorem $1.3$ and Proposition $3.2$ with $\sigma=0$ and $g=0$}) where $\chi_E$ denotes the characteristic function of $E$, $E^C$ is the complement of $E$, and the integral  has to be understood in the principal value sense.

If $E$ is smooth and compactly contained in $\Omega$, let $w$ be a 
 smooth function defined on on $\dE$, with small $L^\infty$ norm. 
We call $E_{w}$ the set whose boundary $\partial E_{w}$ is parametrized by
\begin{equation}\label{eq:Ew}
\partial E_{w}=\{ x+w(x)\nu_E(x)| x\in \dE \}
\end{equation}
where  $\nu_E$ is a normal vector field to $\dE$ exterior to $E$.

The first variation of the $s$-perimeter $\eqref{PersE}$ along these normal perturbations is given by 
\begin{equation}\label{firstvar}
\di_tP_s(E_{tw},\Omega)_{|t=0} =\dfrac{\di}{\di t}_{|t=0}P_s(E_{tw},\Omega)=\int_{\dE} H_s^\Omega(\dE)w,
\end{equation}
see \cite{DDPW}.

In the following, we take $B_1(\xi)$ a ball with center $\xi \in \mathbb{R}^N$ and unit radius, $w\in~C^1(\partial B_1(\xi))$, and we denote by 
$\mathbb{B}(\xi,w)$ the set such that 
\begin{equation}\label{Bdeformato}
\partial \mathbb{B}(\xi,w):=\{y\in \mathbb{R}^N: y=w(x)\nu(x), x\in \partial B_1(\xi)\},
\end{equation}
where $\nu$ is the outer unit normal to $\partial B_1(\xi)$.

Then we let
\begin{equation}\label{P_sxiomega}
S_\xi:=\partial B_1(\xi)\qquad \text{and}\qquad P_{s,\xi}^\Omega(w):=P_s^\Omega (\partial \mathbb{B}(\xi, w)).
\end{equation}
Moreover, for $\beta \in (2s,1)$ and $\varphi \in C^{1,\beta}(\partial \mathbb{B}(\xi,w))$, we  define 
\[
\Big(P_{s,\xi}^\Omega\Big)'(w)[\varphi]:=\int_{\partial \mathbb{B}(\xi,w)}H_s^\Omega (\partial \mathbb{B}(\xi,w))\varphi \di \sigma_w
\]
where  $d \sigma_{w}$ stands for the area element of $\partial \mathbb{B}(\xi, w)$.

Consider next the {\em spherical fractional Laplacian} 
$$
  L_s \varphi (\theta):= P.V.\int_S \dfrac{\varphi(\theta)-\varphi(\sigma)}{|\theta-\sigma|^{N+2s}}\di \sigma, 
$$
where $S=\partial B_1$ and the above integral is understood in the principal value sense.

It turns out that (see e.g. \cite{CFW})
\begin{equation}\label{linearization}
L_s:C^{1,\beta}(S) \rightarrow C^{\beta-2s}(S).\\
\end{equation}
The operator $L_s$ has an increasing sequence of eigenvalues $0=\lambda_0<\lambda_1<\lambda_2<\cdots$ whose explicit expression is given by
\begin{equation}\label{eigenvalues}
\lambda_k:=\dfrac{\pi^{(N-1)/2}\Gamma((1-2s)/2}{(1+2s)2^{2s}\Gamma((N+2s)/2)}\Biggl(\dfrac{\Gamma\Big(\dfrac{2k+N+2s}{2}\Big)}{\Gamma\Big(\dfrac{2k+N-2s-2}{2}\Big)}-\dfrac{\Gamma\Big(\dfrac{N+2s}{2}\Big)}{\Gamma\Big(\dfrac{N-2s-2}{2}\Big)} \Biggr),
\end{equation}
see \cite{S}*{Lemma 6.26}, where $\Gamma$  is the Euler Gamma function. The eigenfunctions are the usual spherical harmonics, i.e.\ one has\ 
\[
L_s \psi=\lambda_k \psi \qquad \text{for every}\; k\in \mathbb{N}\; \text{and}\; \psi \in \mathcal{E}_k,
\]
where $\mathcal{E}_k$ is the space of spherical harmonics of degree $k$ and dimension $n_k = N_k - N_{k-2}$, 
with
$$
  N_k = \frac{(n+k-1)!}{(n-1)! k!}, \quad k \geq 0, \qquad \qquad N_k = 0 \quad k < 0. 
$$
We recall that $n_0=1$ and that $\mathcal{E}_0$ consists of constant functions, whereas $n_1=N$ and $\mathcal{E}_1$ is spanned by the 
restrictions of the coordinate functions in $\R^N$ to the unit sphere $S$. 

\

For sets that are suitable graphs over the unit sphere $S$ of $\R^N$, we have the following 
result concerning fractional mean curvature relative to the whole space, see \cite{CFW}*{Theorem~2.1, Lemma 5.1 and Theorem 5.2 }(see also  formula (1.3) in the latter paper).

\begin{prop}\label{nucleoarmoniche}
Given $\beta \in (2s,1)$, consider the family of functions 
$$
  \Upsilon := \left\{ \varphi \in C^{1,\beta}(S) \; : \; \|\varphi\|_{L^{\infty}(S)} < \frac{1}{2}  \right\}. 
$$
Then the map $\varphi \mapsto H_s^{\R^N}(\partial \mathbb{B}(0,\varphi))$ is a $C^\infty$ function 
from $\Upsilon$ into $C^{\beta - 2 s}(S)$. Moreover, its linearization at $\varphi \equiv 0$ 
is given by 
\begin{equation}\label{lin}
\varphi\longmapsto 2d_{N,s} (L_s-\lambda_1 )\varphi,
\end{equation}
where $\lambda_1$ is defined  in $\eqref{eigenvalues}$ and $d_{N,s}:=\frac{1-2s}{(N-1)|B_1^{N-1}|}$ where $B_1^{N-1}$  is the unit ball in $\mathbb{R}^{N-1}$.
\end{prop}

\

As a consequence of the latter result we have than every function in the kernel of the above 
linearized nonlocal mean curvature is a linear combination of first-order spherical harmonics, 
i.e. if  $w\in \text{Ker}\; (L_s-\lambda_1)$, we have 
\begin{equation}\label{Ts1}
w=\sum_{i=1}^N{\lambda_i Y_i},
\end{equation}
where $\{Y_i\}_{i=1,\cdots,N} \in \mathcal{E}_1$ and $\lambda_i\in \mathbb{R}$.
Therefore, defining 
\begin{equation}\label{eq:space-W}
W:=\Big \{ w\in C^{1,\beta}(S): \int_{S} w \, Y_i = 0 \hbox{ for } i = 1, \dots, N \Big \},
\end{equation}
it follows by Fredholm's theory that $L_s-\lambda_1$ is invertible on $W$. 

\

As a consequence of the above proposition, using a perturbation argument (i.e. an approximate invariance by translation), we deduce also the 
following result, for which we need to introduce some notation. 
Let $\Omega$ be a bounded set in $\R^N$, for $\eps > 0$ let $\Omega_\eps := \frac{1}{\eps} \Omega$.
Fix a compact set $\Theta$ in $\Omega$, and let $\xi \in \frac{1}{\eps} \Theta$. Consider then 
the operator $L_{s, \xi}^{\Omega_\eps}$ corresponding to the linearization of the 
$s$-mean curvature at $B_1(\xi)$ relative to $\Omega_\eps$, namely the 
non-local operator such that 
$$
  \dfrac{\di}{\di t}_{|t=0}H_s^{\Omega_\eps}(\partial \mathbb{B}(\xi, t\varphi))(x)=  (L_{s,\xi}^{\Omega_\eps} \varphi)(x), 
$$
for any $\varphi$ of class $C^{1,\beta}$, $\beta > 2s$. We have then the following result. 

\

\begin{prop}\label{p:invert-orth}
Let $\Omega$, $\Theta$, $\xi$ and $L_{s, \xi}^{\Omega_\eps}$ be as above, and let 
$\beta \in (2s,1)$. Consider the family of functions 
$$
  \Upsilon := \left\{ \varphi \in C^{1,\beta}(S_\xi) \; : \; \|\varphi\|_{L^{\infty}(S_\xi)} < \frac{1}{2}  \right\}. 
$$
Then the map $\varphi \mapsto H_s^{\Omega_\eps}((\partial \mathbb{B}(\xi,\varphi))$ is a $C^\infty$ function 
from $\Upsilon$ into $C^{\beta - 2 s}(S_\xi)$. 
Moreover, if  $W = W_\xi$ is as in \eqref{eq:space-W},   $L_{s, \xi}^{\Omega_\eps}$ 
is invertible with uniformly bounded inverse on $W$.
\end{prop}

\

Given a topological space $M$ and a subset $A \subseteq M$, we recall next the definition and some properties of the Lusternik-Schnirelman category.

\begin{defn}\label{category}\cite{AM}*{Definition $9.2$}
The category of $A$ with respect to $M$, denoted by cat$_M(A)$, is the least integer $k$ such that $A\subseteq A_1 \cup \cdots \cup A_k$ with $A_i$ closed and contractible in $M$ for every $i=1,\cdots, k$.
\end{defn}
We set cat$(\varnothing)=0$ and cat$_M(A)=+\infty$ if there are no integers with the above property. We will use the notation cat$(M)$ for cat$_M(M)$.
\begin{remark}\label{catinclusione}
From Definition \ref{category}, it is easy to see that cat$_M(A)=$ cat$_M (\bar{A})$. Moreover, if $A\subset B\subset M$, we have that cat$_M(A)\le$ cat$_M(B)$, see \cite{AM}*{Lemma $9.6$}.
\end{remark}
Then assuming that 
\begin{equation}\label{M}
M=F^{-1}(0),\; \text{where}\; F\in C^{1,1}(E\subset M,\mathbb{R})\; \text{and}\; F'(u)\neq 0\; \forall \; u\in M,
\end{equation}
we set
\[
\text{cat}_k (M)=\sup \{\text{cat}_M(A) :A \subset M\; \text{and}\; A\; \text{is compact}\}.
\]
Note that if $M$ is compact, cat$_k(M)=$cat$(M)$.
At this point we can state a useful result about the Lusternik-Schnirelman category (see e.g. \cite{AM} 
for the definition of Palais-Smale ((PS)-condition). 
\begin{thm}\label{LSthm}\cite{AM}*{Theorem $9.10$}
Let $M$ be a Hilbert space or a complete Banach manifolds. 
Let $\eqref{M}$ hold, let $J\in C^{1,1}(M,\mathbb{R})$ be bounded from below on $M$ and let $J$ satisfy $(PS)$-condition. Then $J$ has at least cat$_k (M)$ critical points.
\end{thm}

\begin{remark}\label{r:bdry}
If $M$ has boundary, under the same assumptions of Theorem \ref{LSthm} one can still 
find at least cat$_k (M)$ critical points for $J$ provided $\nabla J$ is non 
zero on $\partial M$ and points in the  outward direction.
\end{remark}

\section{Proof of Theorem \ref{mainth}}\label{sec3}

In this section we prove Theorem \ref{mainth} via a finite-dimensional 
reduction. This will determine the location of critical points of the relative $s$-perimeter 
depending on $s$ and the geometry of the domain. One of the main tools 
is the following asymptotic expansion of the relative $s$-perimeter. 
From now on, for every $\eps>0$, we set $\Omega_\eps:=\dfrac{1}{\eps}\Omega$, and we aim to prove that the nonlocal mean curvature $H_s^\Omega$ is {\em sufficiently close} to $H_s^{\mathbb{R}^N}$. Hereafter we will write simply $H_s$ to denote $H_s^{\mathbb{R}^N}$. 

\

\begin{lemma}\label{sviluppoPer_s}
Let $\Theta \subseteq \Omega$ be a fixed compact set. For all $\eps>0$ we consider $ B_1(\bar{x})$ a ball of center $\bar{x}\in \Theta_\eps := \frac{1}{\eps} \Theta$ and with unit radius. Then, for the fractional perimeter, the following expansion holds
\begin{equation}\label{developm}
P_s(B_1(\bar{x}),\Omega_\eps)=P_s(B_1(\bar{x}))-\omega_N \eps^{2s} V_\Omega(\bar{x})+O(\eps^{1+2s})\qquad \text{as}\; \eps \rightarrow 0,
\end{equation}
where $\omega_N$ is the volume of the $N$-dimensional unit ball and
\begin{equation}\label{Vomega}
V_\Omega(x):=\int_{\Omega^C}\dfrac{1}{|\eps x-y|^{N+2s}}\di y. 
\end{equation}
Moreover one has that 
\begin{equation}\label{developm-der}
\nabla_{\bar{x}} P_s(B_1(\bar{x}),\Omega_\eps)= - \omega_N \eps^{2s+1} \nabla_{\bar{x}} V_\Omega(\bar{x}) + O(\eps^{2+2s}).
\end{equation}
\end{lemma}

\begin{proof}
Taking $\eps$ small enough, we can assume $B_1(\bar{x}) \subset \Omega_\eps$. From \eqref{PersE} we have
\begin{equation}
P_s(B_1(\bar{x}),\Omega_\eps)-P_s(B_1(\bar{x}))=-\int_{B_1(\bar{x})}\int_{\mathbb{R}^N\setminus \Omega_\eps} \dfrac{1}{|x-y|^{N+2s}}\di x \di y. 
\end{equation}
If we replace $x$ with $\bar{x}$ in the last integrand, we obtain 
$$
  \dfrac{1}{|x-y|^{N+2s}} = \dfrac{1}{|\bar{x}-y|^{N+2s}} + O\left( \dfrac{1}{|\bar{x}-y|^{N+2s+1}} \right); 
  \qquad x \in B_1(\bar{x}), \quad y \in \mathbb{R}^N\setminus \Omega_\eps. 
$$
Therefore 
$$
  \int_{B_1(\bar{x})}\int_{\mathbb{R}^N\setminus \Omega_\eps} \dfrac{1}{|x-y|^{N+2s}}\di x \di y 
  = \omega_N \int_{\mathbb{R}^N\setminus \Omega_\eps} \dfrac{1}{|\bar{x}-y|^{N+2s}} \di y 
  + \int_{\mathbb{R}^N\setminus \Omega_\eps} \dfrac{O(1)}{|\bar{x}-y|^{N+2s+1}} \di y.
$$
From the latter formulas and a change of variables one then finds 
$$P_s(B_1(\bar{x}),\Omega_\eps)-P_s(B_1(\bar{x})) 
=- \eps^{2s} \omega_N  \int_{\Omega^C}\dfrac{1}{|\bar{x}-y|^{N+2s}}\di y+O(\eps^{1+2s}),
$$
which concludes the proof of \eqref{developm}. Formula \eqref{developm-der} follows in 
a similar manner.   
\end{proof}

\

We evaluate then the deviation of fractional $s$-mean curvature from a constant, 
when is it computed relatively to a large domain. 

\

\begin{lemma}\label{Aepscond}
Let $\beta \in (2s,1)$. For the fractional mean curvature defined in $\eqref{NMC}$, the following expansion holds:
\begin{equation}\label{eq:exp-Hs}
H_s^{\Omega_\eps}(S_\xi)= c_{N,s}+O(\eps^{2s}) 
\qquad \quad \hbox{ in } C^{\beta - 2s}(S_\xi),
\end{equation}
where $c_{N,s}:=H_s(S_\xi)$.
Moreover, one has that 
\begin{equation}\label{eq:exp-der-Hs}
\frac{\partial}{\partial_\xi }H_s^{\Omega_\eps}(S_\xi)= O(\eps^{2s+1})
\qquad \quad \hbox{ in } C^{\beta - 2s}(S_\xi).
\end{equation}
\end{lemma}

\begin{proof}
Using the definition of (relative) $s$-mean curvature we can write
\begin{equation}\label{eq:hs-compl}
H_s^{\Omega_\eps}(S_\xi)=H_s^{\Omega_\eps}(S_\xi)+H_s(S_\xi)-H_s(S_\xi)=c_{N,s}- H_s^{\mathbb{R}^N\setminus \Omega_\eps}(S_\xi),
\end{equation}
where we recall that $c_{N,s}:=H_s(S_\xi)$. Now simply observe that 
$$
(H_s^{\mathbb{R}^N\setminus \Omega_\eps}(S_\xi))(x) = \int_{\mathbb{R}^N\setminus \Omega_\eps} 
\frac{\di y}{|x-y|^{N+2s}} = O(\eps^{2s}). 
$$
Therefore we get
\begin{equation}
H_s^{\Omega_\eps}(S_\xi)=c_{N,s}+O(\eps^{2s}).
\end{equation}
Then, using \eqref{eq:hs-compl}, the 
formula after that, and differentiating with respect to $\xi$, we find 
$$
  \frac{\partial}{\partial_\xi }H_s^{\Omega_\eps}(S_\xi) = - 
  \frac{\partial}{\partial x} \int_{\mathbb{R}^N\setminus \Omega_\eps} 
  \frac{\di y}{|x-y|^{N+2s}} = \int_{\mathbb{R}^N\setminus \Omega_\eps} 
  O(1) \frac{\di y}{|x-y|^{N+2s+1}} = O(\eps^{2s+1}).
$$
We proved \eqref{eq:exp-Hs} and \eqref{eq:exp-der-Hs} in a pointwise sense. 
It is easy however to see that they also hold in the $C^1$ sense on the unit sphere $S_\xi$, 
and therefore also in $C^{\beta - 2s}(S_\xi)$. 
\end{proof}

\

We turn next to a finite-dimensional reduction of the problem, which is possible 
by the smallness of volume in the statement of Theorem \ref{mainth}. We refer 
to  \cite{AM2} for a general treatment of the subject.

\

\begin{prop}\label{LyapunovSchmidt}
Suppose $\Omega$ is a smooth bounded set of $\R^N$,   $\Theta$  a set compactly contained in 
$\Omega$, and let $\beta \in (2s,1)$. For  $\eps>0$ small, let $\xi\in  \Theta_\eps$. Then there exist $w_\eps:S_\xi\rightarrow\mathbb{R}$ in $W$ and $\lambda=~(\lambda_1, \cdots, \lambda_N)\in \mathbb{R}^N$ such that 
\[
Vol(\B(\xi,w_\eps)) = \omega_N; \qquad \quad\ \int_{S_\xi} w_\eps Y_i \di \sigma = 0; \qquad \quad   H^{\Omega_\eps}_s(\partial \B(\xi,w_\eps))=c+\sum_{i=1}^N{\lambda_i Y_i},
\]
where $c \in \R$ is close to $c_{N,s}$ and where $\{Y_i\}_{i=1,\cdots,N}\in \mathcal{E}_1$ (extended 
as zero-homogeneous function in a neighborhood of the unit sphere).
Moreover, there exists $C > 0$ (depending on $\Theta, \Omega, N$ and $s$) such that $\|w_\eps\|_{C^{1,\beta}(S_\xi)}\le C \eps^{2s}$ and such that 
$\|\partial_\xi w_\eps\|_{C^{1,\beta}(S_\xi)}\le C \eps^{2s+1}$.  
\end{prop}

%For all $\eps>0$ we let $\delta>0$,  consider
%\begin{equation}\label{Zeps}
%Z_\eps:=\{B_1(\xi) : \xi \in \Omega_\eps\setminus ( B_\delta(\partial \Omega_\eps)) \}
%\end{equation}
%and
%\begin{equation}\label{Ieps}
%I_\eps(z):=P_s(B_1(\xi))\quad \text{with}\; z\in Z_\eps.
%\end{equation}
%Since we want to prove Proposition \ref{LyapunovSchmidt} through the Lyapunov-Schmidt reduction, we have to check the assumptions in Section \ref{sec2}. For $Z_\eps$ and $I_\eps$ as in $\eqref{Zeps}$ and $\eqref{Ieps}$, the structural assumption $(C^1)$ and $(K)$ are clearly checked.
%Then, from Lemma \ref{Aepscond}, we deduce $(A_\eps)$-condition with $A_\eps=\eps^{1-2s}$ and with an analogous reasoning we deduce $(B_\eps)$ condition with $B_\eps=\eps^{1-2s}$. The $(ND)$ condition follows from Theorem \ref{nucleoarmoniche}. With these conditions in hand, we can prove Proposition \ref{LyapunovSchmidt}:

\

\noindent To make the above formula for $H_s^{\Omega_\eps}$ more precise, we mean  that 
$$
  H^{\Omega_\eps}_s(\partial \B(\xi,w_\eps))(\xi + x(1+w_\eps(x))) =c+\sum_{i=1}^N{\lambda_i Y_i}(x)  
  \qquad \quad \hbox{ for every } x \in S_\xi. 
$$

\

\begin{proof}
Let us denote by $\overline{W}$ the family of functions in $C^{\beta-2s}(S_\xi)$ that are $L^2$-orthogonal, with respect to the \underline{standard} volume element of $S_\xi$, to constants and to the first-order spherical harmonics. Notice that $\overline{W} \subseteq W$, see  \eqref{eq:space-W}. 
Let us consider the two-component function $F_{\overline{W}}  : \Theta_\eps \times C^{1,\beta}(S_\xi) \to C^{\beta-2s}(S_\xi) \times \R$
defined by 
\[
F_{\overline{W}}(\xi,w):= \left( P_{\overline{W}}(H^{\Omega_\eps}_s(\partial \B(\xi,w))), Vol(\B(\xi,w)) - \omega_N \right); 
\qquad \quad w\in W, 
\]
where $\omega_N:=Vol(B_1(\xi))$ and 
$P_{\overline{W}}: C^{\beta-2s}(S_\xi) \mapsto \overline{W}$ the orthogonal $L^2$-projection onto the space $\overline{W}$, 
with respect to the \underline{standard} volume element of $S_\xi$.  
With this notation, we want to find $w \in W$ such that  $F_{\overline{W}}(\xi,w) = (0,0)$.

By Lemma \ref{Aepscond} we have that 
\begin{equation}\label{eq:normFW0}
  F_{\overline{W}}(\xi,0) = (O(\eps^{2s}),0), 
\end{equation}
where the latter quantity is intended to be bounded by $C \epsilon^{2s}$ in the $C^{\beta-2s}(S_\xi)$ sense. 
Here and below, the constant $C$ is allowed to vary from one formula to the other. 

By Proposition \ref{p:invert-orth} and by the fact that 
$$
  d_w Vol(\B(\xi,w))_{|_{w=0}} [\varphi] = \int_{S_\xi} \varphi \, d \sigma, 
$$
we have that $L_\xi :=\nabla_w F_{\overline{W}}(\xi,0)\in Inv(W,\overline{W} \times \R)$ with $\|L_\xi^{-1}\|_{L(\overline{W} \times \R,W)} \le C$.
Hence $F_{\overline{W}}(\xi,w)=(0,0)$ if and only if $F_{\overline{W}}(\xi,0)+L_\xi [w]-L_\xi [w]+F_{\overline{W}}(\xi,w)-F_{\overline{W}}(\xi,0)=0$, 
which can be written as 
\[
w=T_\xi(w):=-L_\xi^{-1}[F_{\overline{W}}(\xi,0)-L_\xi [w]+F_{\overline{W}}(\xi,w)-F_{\overline{W}}(\xi,0)].
\]
Therefore $F_{\overline{W}}(\xi,w)=(0,0)$ if and only if $w$ is a fixed point for $T_\xi$.

Let us show that $T_\xi$ is a contraction in $B_{\overline{C}\eps^{2s}}(\xi)$ for $\overline{C}$ sufficiently large. From the definition of $T_\xi$, 
the above estimate \eqref{eq:normFW0}  and the fact that $$\|L_\xi^{-1}\|_{L(\overline{W} \times \R,W)} \le C,$$ we have 
\begin{equation}
\|T_\xi(0)\|_{C^{1,\beta}(S_\xi)}=\|L_\xi^{-1}[F_{\overline{W}}(\xi,0)]\|_{C^{1,\beta}(S_\xi)}\le C^2\eps^{2s}.
\end{equation}
Then, taking $w_1$ and $w_2\in B_{\bar{C}\eps^{2s}}(\xi)\subseteq W$  it follows that
\begin{equation}
\|T_\xi(w_1)-T_\xi(w_2)\|_{C^{1,\beta}(S_\xi)}\le C\|F_{\overline{W}}(\xi,w_1)-F_{\overline{W}}(\xi,w_2)-L_\xi[w_1-w_2]\|_{C^{1,\beta}(S_\xi)}.
\end{equation}
We notice that  $w \mapsto Vol(\B(\xi,w))$ is a smooth function from the 
metric ball of radius $\frac{1}{2}$ in $C^{1,\beta}(S_\xi)$ into $\R$. 
Thanks also to the smoothness statement in Proposition~\ref{p:invert-orth}, the right hand side in the latter 
formula can be bounded by 
\begin{equation}
\begin{aligned}
F_{\overline{W}}(\xi,w_1)&-F_{\overline{W}}(\xi,w_2)-L_\xi[w_1-w_2]=\int_0^1\Big(\nabla_w F_{\overline{W}}(\xi,w_2+s(w_1-w_2))\\
&-\nabla_w F_{\overline{W}}(\xi,0)\Big) [w_1-w_2] \di s \le C\|w_1-w_2\|_{C^{1,\beta}(S_\xi)}^{2}.
\end{aligned}
\end{equation}
Hence, in $B_{\bar{C}\eps^{2s}}(\xi)\subseteq W$ the Lipschitz constant of $T_\xi$ is $C \bar{C}\eps^{2s}$. So choosing first any $\bar{C}\ge 2C$, and then  $\eps>0$ small enough, we find therefore that $T_\xi$ is a contraction in $B_{\bar{C}\eps^{2s}}(\xi)$. As a consequence, there exists $w_\eps:S_\xi \rightarrow \R$ in $W$ such that $\|w_\eps\|_{C^{1,\beta}(S_\xi)}\le \bar{C} \eps^{2s}$ 
and such that $F_{\overline{W}}(\xi,w_\eps) = (0,0)$.

\

We also recall that the fixed point $w$ can be proved to 
be continuous and differentiable with respect to the parameter $\xi$, (see e.g. \cite{Bressan}, Section 2.6). 
%Now call $Q_{W^\bot}:C^{1,\beta}(\partial B_1) \longmapsto W^{\bot}$ the orthogonal projection in $L^2(S)$. 
%\begin{itemize}
%\item[a)]Since $w_\eps\in W$, we have that $Q_{W^\bot}(w_\eps(\xi))=0$. Differentiating with respect to $\xi$ we obtain
%\[
%\dfrac{\partial Q_{W^\bot}}{\partial \xi}w_\eps(\xi)+Q_{W^\bot}\dfrac{\partial w_\eps}{\partial \xi}=0.
%\]
%By $(K)$-condition it follows that 
%\[
%\Big\|\dfrac{\partial Q_{W^\bot}}{\partial \xi}\Big\|_{C^{1,\beta}(\partial B_1)}\le \tilde{C}. 
%\]
%Therefore 
%\begin{equation}\label{stimaQ}
%\Big\|Q_{W^\bot}\dfrac{\partial w_\eps}{\partial \xi}\Big\|_{C^{1,\beta}(\partial B_1)}\le \Big\|\dfrac{\partial Q_{W^\bot}}{\partial \xi}\Big\|_{C^{1,\beta}(\partial B_1)} \Big\|w_\eps(\xi)\Big\|_{C^{1,\beta}(\partial B_1)}\le \tilde{C} C \eps^{1-2s}.
%\end{equation}
%\item[b)]
Recall that $w_\eps = w_\eps(\xi)$ solves 
\[
Vol(\B(\xi,w_\eps)) = \omega_N\quad \text{and} \quad P_{\overline{W}}(H_s^{\Omega_\eps}(\partial \B(\xi,w_\eps))=0 \quad \qquad \text{for all}\; \xi \in \mathbb{R}^N.
\]
We want next to differentiate the above relations with respect to $\xi$. For this purpose, it is convenient 
to fix an index $i$, and to consider the one-parameter family of centers 
\begin{equation}\label{eq:xi-t}
  \xi(t) = \left(  \xi_1, \dots, \xi_i + t, \dots, \xi_N \right). 
\end{equation}
Our aim is to understand the variation of $\partial \mathbb{B}(\xi_t, w_\eps(\xi_t))$ normal 
to $\partial \mathbb{B}(\xi, w_\eps(\xi))$. The above variation is characterized by a 
translation in the $i$-th component and by a variation of $w_\eps$, which is in the 
radial direction with respect to the center $\xi$. Therefore, letting $\nu_{w_\eps}$ 
denote the unit outer normal vector to  $\partial \mathbb{B}(\xi, w_\eps(\xi))$, the normal variation in $t$ 
(computed at $t = 0$) is given by 
\begin{equation}\label{eq:normal-i-w}
   \nu_{w_\eps} \cdot e_i + \frac{\partial w_\eps(\xi)}{\partial {\xi_i}} (x-\xi) \cdot \nu_{w_\eps}. 
\end{equation}
%
%
%
%Varying the center $\xi$ of the ball corresponds to deforming by a normal variation of the 
%sphere $S$ by a quantity that is a linear combination of the spherical harmonics $Y_i$. Let us 
%now differentiate the above equation with respect to $\xi_i$, $i = 1, \dots, N$. The corresponding infinitesimal translation 
%of $\partial B_1(\xi)$, viewed as a normal variation,
%
Hence we have that 
%
% 
%a unitary variation of $\xi$ and denoting by $$ the corresponding normal variation, 
%we have that 
\begin{equation*}
%\dfrac{\partial P_{\overline{W}}}{\partial \xi}H_s(\partial \B(\xi,w_\eps(\xi)))+
\frac{\partial}{\partial \xi_i} Vol(\B(\xi,w_\eps)) = 0\quad \text{and} \quad P_{\overline{W}} (H_s^{\Omega_\eps})'(\partial \B(\xi,w_\eps(\xi_i)))\Big[\nu_{w_\eps} \cdot e_i + \frac{\partial w_\eps(\xi)}{\partial {\xi_i}} (x-\xi) \cdot \nu_{w_\eps}\Big]=0.
\end{equation*}
Using \eqref{eq:exp-der-Hs} and Proposition \ref{p:invert-orth}  one 
finds from the second equation in the latter formula that $\| v_{i,\eps} \|_{C^{1,\beta}(S_\xi )} 
\leq C \eps^{2s+1}$, where $v_{i,\eps} = P_{\overline{W}}  \partial_{\xi_i} w_\eps$. Since $\dfrac{\partial w_\eps}{\partial \xi_i} \in W$, 
it remains to control then the component of $\partial_{\xi_i} w_\eps$ in the orthogonal complement of 
$\bar{W}$, namely its average. 

Let us write 
$$
  \partial_{\xi_i} w_\eps = v_{i,\eps} + c_{i,\eps} \qquad \text{with}\;   c_{i,\eps} \in \R. 
$$
From a direct computation we have that 
$$
 0 = \frac{\partial}{\partial \xi_i} Vol(\B(\xi,w_\eps)) = \int_{S_\xi} (1 + w_\eps)^{N-1} 
 \left(  v_{i,\eps} + c_{i,\eps}\right) d \sigma. 
$$
Since we know that $\| v_{i,\eps} \|_{C^{1,\beta}(S_\xi )} 
\leq C \eps^{2s+1}$, it follows from the latter formula that also $|c_{i,\eps}| \leq  C \eps^{2s+1}$. 
Therefore one deduces  
\begin{equation}\label{eq:est-der-w}
  \left\| \partial_{\xi_i} w_\eps \right\| _{C^{1,\beta}(S_\xi)}\le C\eps^{2s+1}, 
\end{equation}
which is the desired conclusion, possibly relabelling the constant $C$. 
\end{proof}

\

We next show how to find $\xi$'s so that the Lagrange multipliers $\lambda_i$ 
in the statement of Proposition \ref{LyapunovSchmidt} vanish,  thus obtaining
surfaces with constant relative fractional mean curvature.

\begin{prop}\label{criticoimplicaNMC}
Let $w_\eps:S_\xi \rightarrow \R$ given by Proposition \ref{LyapunovSchmidt}, and for $\xi \in 
\Theta_\eps$ define $\Phi_\xi :=P_s^{\Omega_\eps}(\B(\xi,w_\eps))$. Then, for $\eps>0$ sufficiently small, if $\nabla_\xi {\Phi_{\xi}}_{|_{\xi=\bar{\xi}}}=~0$ for some $\bar{\xi}\in \Theta_\eps$, one has 
\[
H_s^{\Omega_\eps}(\partial \B(\bar{\xi},w_\eps))\equiv c,
\]
where $c=c(\eps, \bar{\xi})$.
\end{prop}

\begin{proof}

%Now call $Q_{W^\bot}:C^{1,\beta}(\partial B_1) \longmapsto W^{\bot}$ the orthogonal projection in $L^2(S)$. 
%\begin{itemize}
%\item[a)]Since $w_\eps\in W$, we have that $Q_{W^\bot}(w_\eps(\xi))=0$. Differentiating with respect to $\xi$ we obtain
%\[
%\dfrac{\partial Q_{W^\bot}}{\partial \xi}w_\eps(\xi)+Q_{W^\bot}\dfrac{\partial w_\eps}{\partial \xi}=0.
%\]
%By $(K)$-condition it follows that 
%\[
%\Big\|\dfrac{\partial Q_{W^\bot}}{\partial \xi}\Big\|_{C^{1,\beta}(\partial B_1)}\le \tilde{C}. 
%\]
%Therefore 
%\begin{equation}\label{stimaQ}
%\Big\|Q_{W^\bot}\dfrac{\partial w_\eps}{\partial \xi}\Big\|_{C^{1,\beta}(\partial B_1)}\le \Big\|\dfrac{\partial Q_{W^\bot}}{\partial \xi}\Big\|_{C^{1,\beta}(\partial B_1)} \Big\|w_\eps(\xi)\Big\|_{C^{1,\beta}(\partial B_1)}\le \tilde{C} C \eps^{1-2s}.
%\end{equation}
%\item[b)]
Recall that $w_\eps = w_\eps(\xi)$ solves 
\[
Vol(\B(\xi,w_\eps)) = \omega_{N}\quad \text{and} \quad P_{\overline{W}}(H_s^{\Omega_\eps}(\partial \B(\xi,w_\eps))=0 \quad \qquad \text{for all}\; \xi \in \mathbb{R}^N.
\]
Since $Vol(\B(\xi,w_\eps)) = \omega_N$ for any choice of $\xi$, it follows that the integral 
over $\partial \mathbb{B}(\xi, w_\eps(\xi))$ of the normal variation vanishes, i.e., recalling 
\eqref{eq:normal-i-w}, we have for $\xi = \bar{\xi}$
\begin{equation}\label{eq:cons-vol}
  \int_{\partial \mathbb{B}(\xi, w_\eps(\xi))} \left[ \nu_{w_\eps} \cdot e_i + \frac{\partial w_\eps(\xi)}{\partial {\xi_i}} (x-\xi) \cdot \nu_{w_\eps} \right] d \sigma_{w_\eps} = 0,   
\end{equation}
where  $d \sigma_{w_\eps}$stands for the area element of $\partial \mathbb{B}(\xi, w_\eps(\xi))$. 

For the same reason, recalling \eqref{firstvar} and \eqref{eq:xi-t}, we have that 
$$
 \frac{\di}{\di t}|_{t=0} P_s^{\Omega_\eps} (\partial \B(\xi(t), w_\eps(\xi(t)))) = 
  \int_{\partial \mathbb{B}(\xi, w_\eps(\xi))}  H_s^{\Omega_\eps}(\partial \B(\bar{\xi},w_\eps))
  \left[ \nu_{w_\eps} \cdot e_i + \frac{\partial w_\eps(\xi)}{\partial {\xi_i}} (x-\xi) \cdot \nu_{w_\eps} \right] d \sigma_{w_\eps}. 
$$
%Let $\delta>0$ and, recalling that
%\[
%Z_\eps=\{B_1(\xi): \xi \in \Omega_\eps\setminus ( B_\delta(\partial \Omega_\eps))\},
%\]
%we denote
%\[
%\tilde{Z}_\eps:=\{\B(\xi,w): \xi \in \Omega_\eps\setminus ( B_\delta(\partial \Omega_\eps))\}.
%\]
%Since, from Proposition \ref{LyapunovSchmidt}, we have $\|w_\eps\|_{C^{1,\beta}(\partial B_1)}\le \eps$ for all $k\in \mathbb{N}$, it follows that
%\[
%T_{B_1(\xi)}Z_\eps \backsimeq T_{\B(\xi,w)}\tilde{Z}_\eps,
%\]
%in the sense that the two tangent planes are almost parallel.
%Moreover, by construction 
%\begin{equation}
%\nabla_\xi \Phi_{\xi} \in T_{\B(\xi,w_\eps)}\tilde{Z}_\eps,\qquad \forall \;\xi \in \mathbb{R}^N
%\end{equation}
%and, from the hypothesis $\nabla_\xi \Phi_{\bar{\xi}}\, \bot \, \tilde{Z}_\eps$, so 
%\[
%\nabla_\xi \Phi_{\bar{\xi}}=0,
%\]
%that is our thesis.
By our choice of $\bar{\xi}$ we have that, for all $i = 1, \dots, N$
$$
  \frac{\partial}{\partial \xi}_i|_{\xi = \bar{\xi}} \Phi_\xi = 0. 
$$
Recalling also that by Proposition \ref{LyapunovSchmidt}, 
$H^{\Omega_\eps}_s(\partial \B(\xi,w_\eps))=c+\sum_{i=1}^N{\lambda_i Y_i}$ (see Section \ref{sec2} 
for the definition of the first-order sphereical harmonics $Y_i$),  
from \eqref{eq:cons-vol} we have that for all $i = 1, \dots,  N$ 
\begin{equation}\label{eq:ij}
0 = \int_{\partial \mathbb{B}(\xi, w_\eps(\xi))}  \left( \sum_{j=1}^N{\lambda_j Y_j} \right)
   \left[ \nu_{w_\eps} \cdot e_i + \frac{\partial w_\eps(\xi)}{\partial {\xi_i}} (x-\xi) \cdot \nu_{w_\eps} \right] d \sigma_{w_\eps}. 
\end{equation}
Notice that by the estimates on $w_\eps$ and $\partial_\xi w_\eps$ in Proposition \ref{LyapunovSchmidt} and by 
the fact that $\nu \cdot {\bf e_i} = Y_i$ on the unit sphere $S$, 
one has 
$$
 \int_{\partial \mathbb{B}(\xi, w_\eps(\xi))} Y_j 
     \left[ \nu_{w_\eps} \cdot e_i + \frac{\partial w_\eps(\xi)}{\partial {\xi_i}} (x-\xi) \cdot \nu_{w_\eps} \right] d \sigma_{w_\eps} 
     = \delta_{ij} + o_\eps(1); \qquad i, j = 1, \dots, N.
$$
Therefore the system \eqref{eq:ij} implies the vanishing of all $\lambda_j$'s, 
%
%
%
% and 
%$\nu$ the outer unit normal to $\partial B_1(\xi)$. Hence we have that 
%%
%Recalling the proof of the  previous proposition and the definition of $\Phi_\xi$, we have that 
%$$
%  \frac{\partial}{\partial \xi_j} \Phi_\xi = \int_{\partial \mathbb{B}(\xi,w)} H_s^{\Omega_\eps} \Big[Y_j+
%  \dfrac{\partial w_\eps}{\partial \xi_j}\Big] 
%  d \sigma. 
%$$
%Using the fact that $H^{\Omega_\eps}_s(\partial \B(\xi,w_\eps))=c+\sum_{i=1}^N{\lambda_i Y_i}$ we 
%find that 
%$$
%  0i  = \int_{\partial \mathbb{B}(\xi,w)} \left[ c+\sum_{i=1}^N{\lambda_i Y_i} \right] 
% \Big[Y_j+\dfrac{\partial w_\eps}{\partial \xi_j}\Big] \di \sigma. 
%$$
%Notice that by the constant area condition one has $\displaystyle \int c \dfrac{\partial w_\eps}{\partial \xi} d \sigma = 0$, 
which gives the desired conclusion. 
\end{proof}

%\begin{lemma}\label{stimacurvaturaPalladeformata}
%Let $w_\eps$ be as in Proposition \ref{criticoimplicaNMC}. The following expression holds:
%\[
%H_s(\partial \B(\xi,tw_\eps))=const.+O(\eps^{1-2s}).
%\]
%\end{lemma}
%
%\begin{proof}
%It is similar to the proof of Lemma \ref{Aepscond}.
%\end{proof}

\

The next step is to show that fractional perimeter of $B_1(\xi)$ is sufficiently close to fractional perimeter of the deformed ball $\B(\xi,w_\eps)$, 
also when differentiating with respect to $\xi$. 

\begin{prop}\label{persdevelopment}
Let $w_\eps$ be as Proposition \ref{criticoimplicaNMC}.  The following Taylor expansion holds:
\begin{equation}\label{sviluppoPerimetroPalladeformata}
P_s^{\Omega_\eps}(\B(\xi,w_\eps))=P_s^{\Omega_\eps}(B_1(\xi))+O(\eps^{4s}).
\end{equation}
Moreover one has 
\begin{equation}\label{der-sviluppoPerimetroPalladeformata}
\frac{\partial}{\partial \xi_i} P_s^{\Omega_\eps}(\B(\xi,w_\eps))= \frac{\partial}{\partial \xi_i} 
P_s^{\Omega_\eps}(B_1(\xi))+O(\eps^{1+4s}).
\end{equation}
\end{prop}
\begin{proof}
Thanks to the first statement of Lemma \ref{Aepscond}, following the notation in Section \ref{sec2}, we get that
\begin{equation}
\begin{aligned}
P_s^{\Omega_\eps}(\mathbb{B}(\xi, w_\eps))&=P_s^{\Omega_\eps}(B_1(\xi))+(P_s^{\Omega_\eps})'[w_\eps]+P_s^{\Omega_\eps}(\mathbb{B}(\xi,w_\eps))-
(P_s^{\Omega_\eps})'[w_\eps]-P_s^{\Omega_\eps}(B_1(\xi))\\
&=P_s^{\Omega_\eps}(B_1(\xi))+O(\eps^{4s})+\int_0^1\Big((P_s^{\Omega_\eps})'(t \, w_\eps)-(P_s^{\Omega_\eps})'(0)\Big)[w_\eps]\di t,
\end{aligned}
\end{equation}
where $(P_s^{\Omega_\eps})'$ is defined as in the formula after \eqref{firstvar}.

Using the fact that the $s$-mean curvature is smooth, we deduce then that 
$$
\int_0^1\Big((P_s^{\Omega_\eps})'(t \, w_\eps)-(P_s^{\Omega_\eps})'(0)\Big)[w_\eps]\di t=O(\eps^{4s}),
$$
so the last two formulas imply \eqref{sviluppoPerimetroPalladeformata}.

To prove \eqref{der-sviluppoPerimetroPalladeformata}, we use the estimate $\|\partial_\xi w_\eps\|_{C^{1,\beta}(S_\xi)}\le C \eps^{2s+1}$ from Proposition \ref{LyapunovSchmidt}. Calling $\tau_i$ the quantity in  \eqref{eq:normal-i-w} and 
recalling
the notation from Section \ref{sec2}, we write that 
$$
  \frac{\partial}{\partial \xi_i} P_s^{\Omega_\eps}(\B(\xi,w_\eps)) = (P_s^{\Omega_\eps})'(w_\eps) [\tau_i]. 
$$
Taylor-expanding the latter quantity we can write that 
\begin{equation}
\begin{aligned}
\frac{\partial}{\partial \xi_i} P_s^{\Omega_\eps}(\B(\xi,w_\eps)) & =   (P_s^{\Omega_\eps})'(0)  [\tau_i]
+ (P_s^{\Omega_\eps})''(0)  [\tau_i] \\ 
& =  \frac{\partial}{\partial \xi_i} 
P_s^{\Omega_\eps}(B_1(\xi)) + O(\eps^{1+4s}).
\end{aligned}
\end{equation}
This concludes the proof. 
\end{proof}

\

\begin{proof}[Proof of Theorem \ref{mainth}]
Suppose $x_0$ is a strict local extremal of $V_\Omega$, without loss of generality a minimum. 
Then there exists an open set $\Upsilon \subset\subset \Omega$ such that 
$V_\Omega(x_0) < \inf_{\partial \Upsilon} V_\Omega - \delta$ for some $\delta > 0$. 
  Let $\Phi_{\xi}$ be defined as in Proposition \ref{criticoimplicaNMC}: by the 
  estimates \eqref{developm} and \eqref{sviluppoPerimetroPalladeformata} 
  it follows that 
\begin{equation}\label{eq:exp-Phi-V}
  \Phi_{\bar{x}} = 
  P_s^{\R^N}(B_1(\bar{x}))-\omega_N \eps^{2s} V_\Omega(\eps \bar{x})+O(\eps^{1+2s}),
\end{equation}
which implies that for $\eps$ sufficiently small 
$$
  \Phi_{\frac{x_0} \eps} < \inf_{\frac{1}{\eps}  \, \partial \Upsilon} \Phi.  
$$
As a consequence $\Phi_{\cdot}$ attains a minimum in the dilated domain $\frac{1}{\eps} 
\Upsilon$, and the conclusion follows from Proposition \ref{criticoimplicaNMC}.

Suppose now that $x_0$ is  a non-degenerate critical point of $V_\Omega$. 
Recalling the definition and properties of topological degree (see e.g. Chapter 3 in \cite{AM}), 
from \eqref{developm-der} and \eqref{der-sviluppoPerimetroPalladeformata} 
one can find an open set $\tilde{\Upsilon} \subset\subset \Omega$ such that 
$$
  \hbox{deg} \left( \nabla \Phi,  \frac{1}{\eps}\tilde{\Upsilon}, 0 \right) \neq 0.
$$
This implies that $\Phi_{\xi}$ has a critical point in $\frac{1}{\eps} 
\tilde{\Upsilon}$, and the conclusion again follows  from Proposition \ref{criticoimplicaNMC}.

Since in both cases the sets $\Upsilon$ and $\tilde{\Upsilon}$ containing $x_0$ can be taken arbitrarily small, 
the localization statement in the theorem is also proved.  
\end{proof}

\

\begin{remark}
From \cite{AM2}*{Theorem $2.24$} one has a relation between the Morse index of a critical point as found 
in Proposition \ref{criticoimplicaNMC} and the Morse index of the corresponding critical point of $\Phi$. 
In our case, since round spheres are global minimizers for the $s$-perimeter relative to $\R^N$, 
these two indices coincide. 
\end{remark}

\

To prove Corollary \ref{Maincoro}, we need the following Lemma.

\begin{lemma}\label{comportamentoVomega}
For all $x\in \partial \Omega$ one has 
\[
\lim_{y\rightarrow x}V_\Omega(y)=+\infty, 
\]
and
\[
\lim_{\Omega \ni y\rightarrow x}\nabla V_\Omega(y)\cdot \nu (x)=+\infty,
\]
where $\nu$ denotes the outer unit normal to $\partial \Omega$.
\end{lemma}

\begin{proof}
Letting $d:=\dist(x,\partial \Omega)$ for  $x\in \Omega$, thanks to the change of variables $x'=\frac{x}{d}$, we get that
\begin{equation}
V_\Omega (x)=\int_{\Omega^C}\dfrac{1}{|x-y|^{N+2s}}\di y=\int_{(\Omega/d)^C} \frac{1}{|dx'-y'|^{N+2s}}\di y',
\end{equation}
from which, setting $\mathbb{R}_+^N=\{x\in \mathbb{R}^N : x_N>0\}$, we have
\[
\int_{(\Omega/d)^C} \frac{1}{|dx'-y'|^{N+2s}}\di y' \rightarrow \int_{(\mathbb{R}_+^N)^C} \frac{1}{|y'|^{N+2s}}\di y'<+\infty 
\qquad \hbox{ if } d \rightarrow 0, 
\]
i.e.\ $V_\Omega$ behaves asymptotically as $d^{-N-2s}$ when $d \to 0$. With a similar proof, one finds that the component of $\nabla V_\Omega$ 
normal to $\partial \Omega$ behaves as $d^{-N-2s-1}$.
\end{proof} 

\

\begin{proof}[Proof of Corollary \ref{Maincoro}]
Given $\delta > 0$ small enough, let us define the set $\Omega^\delta \subseteq \Omega$ 
by 
$$
  \Omega^\delta = \left\{ x \in \Omega \; : \; d(x, \partial \Omega) > \delta \right\}. 
$$
From Remark \ref{comportamentoVomega} we have
\[
(\nabla V_\Omega, \nu_{\Omega^\delta}) >  0 \quad \text{on}\; \partial \Omega^\delta.
\] 
As in the proof of Theorem \ref{mainth}, it turns out that 
\[
(\nabla  \Phi_{\cdot}, \nu_{\frac{1}{\eps} \, \Omega^\delta}) >  0 \quad \text{on}\; 
\partial \frac{1}{\eps} \left(  \Omega^\delta \right).
\]
Clearly, since $\bar{\Omega}$ is compact, the $(PS)$-condition holds. So the conclusion follows from Theorem \ref{LSthm} and Remark \ref{r:bdry}.
\end{proof}

\

\begin{remark}\label{ossne}
It is interesting to see how the geometry of the domain (and not just the topology, as in Corollary \ref{Maincoro}) plays a  role in order to obtain either uniqueness of 
multiplicity of solutions. 

In the Appendix  we will prove uniqueness for the unit ball $B_1$, i.e.\ we will show that $V_{B_1}$ has a unique critical point at the origin which is a non-degenarate minimum.

Secondly, we will give an example of dumb-bell domain, topologically equivalent to a ball, such that the reduced functional $\Phi_\xi$ (defined as in Proposition \ref{criticoimplicaNMC}) has at least three critical points, while Corollary \ref{Maincoro} would give us only one solution.

\section{Proof of Theorem \ref{Mainthm2}}\label{sec4}

Let us consider a bounded open set with smooth boundary $\Omega \subseteq \mathbb{R}^N$,  and $s\in (0,1/2)$.

First of all we point out that, using the direct method of Calculus of Variations and the Sobolev embeddings (which hold for fractional spaces too, see e.g. \cite{DNPV}), it is easy to show that there exist minimizers for 
\begin{equation}\label{esistenza}
\{P_s(E,\Omega), |E|=m \}\quad m\in (0,+\infty).
\end{equation}
Our goal is to prove that minimizers exist also relatively to half-spaces, and to  
 characterize them to some extent.  
 
Let $s\in (0,1/2)$ and $E\subset \mathbb{R}^N$ be a measurable set: recall from $\eqref{PersE}$ that
\begin{equation}\label{F}
P_s(E,\mathbb{R}_+^N):=\int_E \int_{\mathbb{R}_+^N \setminus E} \frac{\di x \di y}{|x-y|^{N+2s}},
\end{equation}
where $\mathbb{R}_+^N=\{ x\in \mathbb{R}^N : x_N>0\}$ is the half-space.
We begin by studying minimizers of
\begin{equation}\label{P}
\{P_s(E,\mathbb{R}_+^N):\;E \subseteq B_R^+,\; |E|=m \}\quad m\in (0,+\infty),
\end{equation}
with $B_R^+:=B_R \cap \mathbb{R}_+^N$  denoting the half ball of large radius $R>0$ centred at the origin.
Without loss of generality   we can assume that $m=1$ and, since we look for minimizers in a half-ball, we can assume that $E$ is closed.
With completely similar arguments, one can also prove the following result.

\begin{prop}
Problem $\eqref{P}$ admits a minimizer.
\end{prop}

%\begin{remark}\label{Oss1}
%The functional $\eqref{F}$ is invariant with respect to horizontal translations, i.e.\
%\begin{equation}
%P_s(E,\mathbb{R}_+^N)=P_s(E+\tilde{z},\mathbb{R}_+^N)\quad \text{for all}\; \tilde{z}=(z_1,\cdots, z_{N-1},0).
%\end{equation}
%\end{remark}
%
%\begin{proof}
%It is trivial consequence of the following change of variables
%\begin{equation*}
%\begin{aligned}
%E &\ni x \longmapsto x=x'+\tilde{z}\in E+\tilde{z},\\
%\mathbb{R}_+^N \setminus E &\ni y \longmapsto y'=y+\tilde{z}\in \mathbb{R}_+^N \setminus (E+\tilde{z}).
%\end{aligned}
%\end{equation*}
%\end{proof}
%
%\begin{remark}
%If $E$ is a minimizer of $\eqref{P}$, for all $\tilde{z}=(z_1,\cdots, z_{N-1},0)$ also the set $E+\tilde{z}$ is a minimizer for $\eqref{P}$.
%\end{remark}
%
%\begin{proof}
%Assume by contradiction that $E+\tilde{z}$ is not a minimizer. Then, there exists $A\subseteq \mathbb{R}^N$, $|A|=1$ such that
%\begin{equation*}
%P_s(A,\mathbb{R}_+^N)<P_s(E+\tilde{z},\mathbb{R}_+^N).
%\end{equation*}
%From Remark \ref{Oss1} we know that
%\[
%P_s(E+\tilde{z},\mathbb{R}_+^N)=P_s(E,\mathbb{R}_+^N),
%\]
%so we contradict the minimality of $E$.
%\end{proof}

\

We have next the following lemma. 

\

\begin{lemma}\label{l:touch}
If $E$ is a minimizer for $\eqref{P}$, then $E$ intersects the plane $\{z_N=0\}$.
\end{lemma}

\begin{proof}
By contradiction suppose that $E$, (which, we recall,  can be taken closed), does not intersect the plane $\{z_N=0\}$. 
We consider then 
the shifted set $E-\lambda e_N$, where $(e_1,\cdots, e_N)$ is the canonical basis of $\mathbb{R}^N$, $\lambda=\dist (E, \{z_N=~0\})>~0$ and we consider
\begin{equation*}
P_s(E-\lambda e_N,\mathbb{R}_+^N)=\int_{E-\lambda e_N}\int_{\R \setminus (E-\lambda e_N)}\frac{\di x \di y}{|x-y|^{N+2s}}.
\end{equation*}
Using the following change of variables (i.e., translating downwards the set $E$ by $\lambda \overrightarrow{e_N}$)
\begin{equation*}
\begin{aligned}
E-\lambda e_N&\ni x \longmapsto x'=x+\lambda e_N \in E,\\
(E-\lambda e_N)^C& \ni y \longmapsto y'=y-\lambda e_N \in E^C,
\end{aligned}
\end{equation*}
where $(E-\lambda e_N)^C$ and $E^C$ are the complements of the sets $E-\lambda e_N$ and $E$ respectively, we have
\[
P_s(E-\lambda e_N,\mathbb{R}_+^N)=\int_E \int_{\R\setminus E}\frac{\di x \di y}{|x+\lambda e_N-y+\lambda e_N|^{N+2s}}<P_s(E,\mathbb{R}_+^N).
\]
This is in contradiction to the minimality of $E$ for $\eqref{P}$.
\end{proof}

Now we want to show other basic properties of minimizers for $\eqref{P}$. To see these, we premise a useful
\begin{defn}
Given a function $u: \mathbb{R}^{N}\rightarrow \mathbb{R}^+$, we define $u^*:\mathbb{R}^{N}\rightarrow \mathbb{R}^+$ the radially symmetric rearrangement of $u$ with respect to \ $x_N$ so that, given $x_N>0$, $t>0$, the superlevel set $\{u^*(\cdot, x_N)>t \}$ is a ball $B$ in $\mathbb{R}^{N-1}$ centered at the origin and
\begin{equation*}
|\{u^*(\cdot,x_N)>t \}|=|\{u(\cdot, x_N)>t \}|,
\end{equation*}
see Figure \ref{fig:5}. 

If $u=\chi_E$, we call $E^*$ the ball such that $\chi_{E^*}=(\chi_E)^*$.
\end{defn}

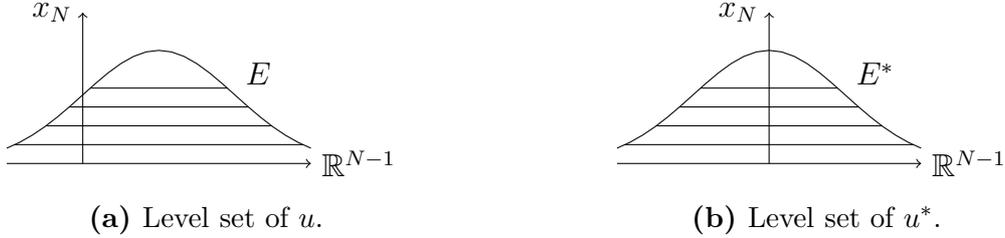
\begin{figure}[h]
  \centering
  \begin{subfigure}{0.45\columnwidth}
    \centering
    \begin{tikzpicture}
      \draw[->] (-2,0) -- +(4,0) node[right]{$\mathbb{R}^{N-1}$};
      \draw[->] (-1,0) -- +(0,2) node[left]{$x_N$};
      \begin{scope}
        \draw plot[domain=-2:2] (\x,{1.5*exp(-0.5*\x*\x)}) [clip];
        \draw (-2,1.00) -- +(4,0);
        \draw (-2,0.75) -- +(4,0);
        \draw (-2,0.50) -- +(4,0);
        \draw (-2,0.25) -- +(4,0);
      \end{scope}
      \draw (1,0.9) node [above right] {$E$};
    \end{tikzpicture}
    \caption{Level set of $u$.}\label{fig:5a}
  \end{subfigure}\qquad
  \begin{subfigure}{0.45\columnwidth}
    \centering
    \begin{tikzpicture}
      \draw[->] (-2,0) -- +(4,0) node[right]{$\mathbb{R}^{N-1}$};
      \draw[->] (0,0) -- +(0,2) node[left]{$x_N$};
      \begin{scope}
        \draw plot[domain=-2:2] (\x,{1.5*exp(-0.5*\x*\x)}) [clip];
        \draw (-2,1.00) -- +(4,0);
        \draw (-2,0.75) -- +(4,0);
        \draw (-2,0.50) -- +(4,0);
        \draw (-2,0.25) -- +(4,0);
      \end{scope}
      \draw (1,0.9) node [above right] {$E^*$};
    \end{tikzpicture}
    \caption{Level set of $u^*$.}\label{fig:5b}
  \end{subfigure}
  \caption{The radially symmetric rearrangement of $u$.}\label{fig:5}
\end{figure}

\begin{defn}
Given a function $u:\mathbb{R}^N \rightarrow \mathbb{R}^+$, we define $\hat{u}:\mathbb{R}^N \rightarrow \mathbb{R}^+$ to be the decreasing rearrangement of $u$ with respect to \ $x_N$: given $x'>0$, $t>0$, $\{x_N:\hat{u}(x',x_N)>t\}\subseteq~\mathbb{R}^+$ is a segment of the form $[0,\alpha)$ with $\alpha:=|\{x_N:\hat{u}(x',x_N)>t\}|$, as in Figure \ref{fig:10}.

If $u=\chi_E$, we call $\hat{E}$ the set such that $\chi_{\hat{E}}=\hat{(\chi_E)}$. Notice that $\partial \hat{E}$ is a graph in the direction $e_N$.
\end{defn}

\begin{figure}[h]
  \centering
  \begin{subfigure}[b]{0.45\columnwidth}
    \centering
    \begin{tikzpicture}
      \draw[->] (-2,0) -- +(4,0) node [right] {$\mathbb{R}^{N-1}$};
      \draw[->] (-0.5,0) -- +(0,2) node [left] {$x_N$};
      \begin{scope}
        \draw[clip,use Hobby shortcut,out angle=90,in angle=90] (-0.7,0) .. (-1.4,0.8) .. (-1.9,1.2) .. (-1.4,1.6) .. (0,1.2) .. (0.7,1.2) .. (1.3,1.0) .. (1.9,0);
        \foreach \x in {-2,-1.5,...,+2}
          \draw (\x,0) -- +(0,3);
      \end{scope}
      \begin{scope}
        \draw[clip] (1,2) circle [radius=0.3];
        \foreach \x in {-2,-1.5,...,+2}
          \draw (\x,0) -- +(0,3);
      \end{scope}
    \end{tikzpicture}
    \caption{Level set of $u$.}\label{fig:10a}
  \end{subfigure}\qquad
  \begin{subfigure}[b]{0.45\columnwidth}
    \centering
    \begin{tikzpicture}
      \draw[->] (-2,0) -- +(4,0) node [right] {$\mathbb{R}^{N-1}$};
      \draw[->] (-0.5,0) -- +(0,2) node [left] {$x_N$};
      \begin{scope}
        \draw[clip,use Hobby shortcut,out angle=90,in angle=260] (-1.9,0) .. (-1.5,0.6) .. (-0.9,0.9) .. (-0.7,1.4) -- (-0.7,1.4) [out angle=0,in angle=180] .. (0,1.2)  .. (0.7,1.2) -- (0.7,1.2) [out angle=90,in angle=90] .. (1.0,1.7) .. (1.3,1.0) -- (1.3,1.0) [out angle=-30,in angle=90] .. (1.9,0);
        \foreach \x in {-2,-1.5,...,+2}
          \draw (\x,0) -- +(0,3);
      \end{scope}
    \end{tikzpicture}
    \caption{Level set of $\hat{u}$.}\label{fig:10b}
  \end{subfigure}
  \caption{The decreasing rearrangement of $u$.}\label{fig:10}
\end{figure}
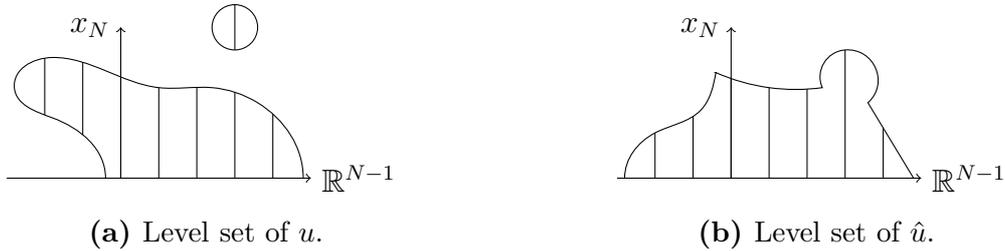

With these definitions at hand, we can show a first property of minimizers of $\eqref{P}$: 
\begin{lemma}\label{minimosimmetrico}
If $E$ is a minimizer of $\eqref{P}$, we have that
\begin{equation*}
P_s(E^*,\mathbb{R}_+^N)\le P_s(E,\mathbb{R}_+^N)
\end{equation*}
and the equality holds if and only if $E=E^*$.
\end{lemma}

\begin{proof}
Proceeding as in \cite{NPS}, we define
\[
\mathcal{H}^s(\mathbb{R}_+^N):=\{u\in L^2(\mathbb{R}_+^N): [u]_{\mathcal{H}^s(\mathbb{R}_+^N)}<+\infty \},
\]
where
\begin{equation}\label{seminorma}
[u]_{\mathcal{H}^s(\mathbb{R}_+^N)}^2:=\inf \Big \{ \int_{\mathbb{R}_+^N \times \mathbb{R}^+}(|\nabla v|^2+|\partial_y v|^2)y^{1-2s}\di x \di y:v\in H_{\text{loc}}^1(\mathbb{R}_+^N \times \mathbb{R}^+), v(\cdot,0)=u(\cdot) \Big \}.
\end{equation}
The space $\mathcal{H}^s(\mathbb{R}_+^N)$ is endowed with the Hilbert norm
\[
\|u\|_{\mathcal{H}^s(\mathbb{R}_+^N)}^2=\|u\|_{L^2(\mathbb{R}_+^N)}^2+[u]_{\mathcal{H}^s(\mathbb{R}_+^N)}^2.
\]
According to $\eqref{seminorma}$ we get
\begin{equation}\label{psesteso}
P_s(E,\mathbb{R}_+^N)=\frac{1}{2}\inf \Big \{ \int_{\mathbb{R}_+^N\times \mathbb{R}^+} (|\nabla_x v|^2+v_y^2)y^{1-2s}\di x \di y : v \in H_{\text{loc}}^1(\mathbb{R}_+^N\times \mathbb{R}^+), v(\cdot, 0)=\chi_E(\cdot)\Big \},
\end{equation}
and we define
\[
H^1(\mathbb{R}_+^N \times \mathbb{R}^+,y^{1-2s}\di y):=\Big \{v\in H_{\text{loc}}^1(\mathbb{R}_+^N \times \mathbb{R}^+): \int_{\mathbb{R}_+^N \times \mathbb{R}^+}(|v|^2+|\nabla_x v|^2+|\partial_y v|^2)y^{1-2s}\di x \di y <\infty \Big \}.
\]
For all $v\in H^1(\mathbb{R}_+^N \times \mathbb{R}^+,y^{1-2s}\di y)$, we set $v^*(\cdot,y)=[v(\cdot,y)]^*$. Then
\begin{itemize}
\item[a)]since the symmetrization preserves characteristic functions, we have that
\begin{equation}\label{thirdterm}
(\chi_E(\cdot))^*=\chi_{E^*}(\cdot);
\end{equation}
\item[b)]from \cite{B}*{Theorem $1$} we get 
\begin{equation}\label{secondterm}
\int_{B_R^+ \times \mathbb{R}^+}(|\nabla_x v^*|^2 +(v_y^*)^2)y^{1-2s}\di x \di y \le \int_{B_R^+ \times \mathbb{R}^+}(|\nabla_x v|^2+v_y^2) y^{1-2s}\di x \di y.
\end{equation}
\end{itemize}
Hence combining $\eqref{psesteso}$, $\eqref{thirdterm}$ and $\eqref{secondterm}$ we deduce the desired conclusion.
\end{proof}
In a similar way, we obtain the following
\begin{lemma}\label{grafico}
Let $E$ be a minimizer of $\eqref{P}$. Then 
\begin{equation*}
P_s(\hat{E},\mathbb{R}_+^N)\le P_s(E,\mathbb{R}_+^N)
\end{equation*}
and the equality holds if and only if $E=\hat{E}$.
\end{lemma}

\begin{proof}
Proceeding as in Lemma \ref{minimosimmetrico} and  setting $\hat{v}(\cdot, y)=\hat{[v(\cdot,y)]}$, we have that
\begin{equation}\label{primotermine}
(\hat{\chi_E(\cdot)})=\chi_{\hat{E}}(\cdot),
\end{equation}
and from \cite{B}*{Theorem $1$} we get
\begin{equation}\label{secondotermine}
\int_{B_R^+ \times \mathbb{R}^+}(|\nabla_x \hat{v}|^2 +(\hat{v_y})^2)y^{1-2s}\di x \di y \le \int_{B_R^+ \times \mathbb{R}^+}(|\nabla_x v|^2+v_y^2) y^{1-2s}\di x \di y.
\end{equation}
Recalling $\eqref{psesteso}$ and using $\eqref{primotermine}$ and $\eqref{secondotermine}$ we conclude the proof.
\end{proof}
\begin{remark}
Note that from these two symmetrizations we obtain a connected minimizer for $\eqref{Phalf}$. 
\end{remark}
We next prove an estimate on the diameter of a set minimizing $\eqref{P}$:
\begin{thm}\label{stimadiametro}
There exists a positive constant $C_1$ such that, for $R$ large, 
uf $E$ is a minimizer of $\eqref{P}$, then
\begin{equation}\label{Diamestimate}
|\text{diam}\; E| \le C_1, 
\end{equation}
with diam $E$ denoting the diameter of the set $E$.
\end{thm}

\begin{proof}
Thanks to Lemma \ref{minimosimmetrico} and Lemma \ref{grafico}, we can assume that there exists $H>0$ such that
\begin{equation}
[0,He_N]\subseteq~E
\end{equation}
and that, for all $t>0$,
\begin{equation}
E_t=E \cap \{x_N=t\}=B_{R(t)}.
\end{equation}
We fix $r_0 > 0$, and we divide the interval $[0,He_N]$  into $M$ sub-intervals of length at most $2r_0$, 
so $M \leq [ \frac{H}{2r_0}  ]+1$. For every sub-interval we consider its center $x^i$, $i=1,\cdots, M$. 

From \cite{MV}*{Theorem $1.7$} we have that, if $r_0$ is sufficiently small depending on $N$ and $s$, 
there exists $C_0>0$ such that for every $x^i$ there exists  a ball $B_{r_0}(x^i)$ with center at $x^i$ and radius $r_0$ such that
\begin{equation*}
|E\cap B_{r_0}(x^i)| \ge \frac{r_0^N}{C_0}>0\quad \text{for all}\;  i=1, \cdots, M. 
\end{equation*}
From this it follows that 
\begin{equation*}
1= |E|\ge  \Big|\frac{H}{2r_0}\Big |\cdot\frac{r_0^N}{C_0}, 
\end{equation*}
and hence
\begin{equation}\label{stimaH}
|H|\le \frac{2C_0}{r_0^{N-1}}.
\end{equation}
We proceed similarly to estimate $R(t)$ for all $t>0$, obtaining that
\begin{equation}\label{stimaEt}
|R(t)|\le \frac{2C_0}{r_0^{N-1}}\qquad \text{for all }t>0.
\end{equation}
Combinig $\eqref{stimaH}$ and $\eqref{stimaEt}$, we deduce the assertion.
\end{proof}

\medskip 

As a corollary we get that a minimizer for $\eqref{P}$ is a minimizer for $\eqref{Phalf}$:
\begin{coro}\label{minimolibero}
Let $E$ be a minimizer of $\eqref{P}$. If $R>2 C_1$, with $C_1$ given by Theorem \ref{stimadiametro}, 
then $E$  is a free minimizer, i.e.\
\begin{equation*}
\bar{E}\cap \partial B_R^+=\varnothing.
\end{equation*}
\end{coro}

Finally we prove the following result: 
\begin{prop}
Let $E$ be a minimizer of $\eqref{P}$. Then $\partial E$ is of class $C^\infty$.
\end{prop}

\begin{proof}
From Lemma \ref{grafico} we know that $\partial E$ is a graph in the $x_N$-direction. Then, \cite{BFV}*{Corollary $3$} implies that $\partial E$ is of class $C^\infty$ outside a closed singular set of Hausdorff dimension $N-8$.

Assume by contradiction that the singular set is nonempty. Since by Lemma \ref{minimosimmetrico} $E$ is radially symmetric, the singular set has to be its highest point in the $x_N$ direction. Moreover,
the blow-up of $E$ centered at the singular point is a singular symmetric cone $C$ contained in a halfspace. By density estimates (see \cite{MV}*{Theorem $1.7$}), we also know that $C \neq \varnothing$, hence $C$ is a Lipschitz cone. By \cite{FV}*{Theorem $1$} we then get that $C$ is a halfspace, hence it cannot be singular, and $\partial E$ is of class $C^\infty$. 
\end{proof}
\begin{remark}
It would be interesting to know whether minimizers, or even critical points, of the functional in $\eqref{Phalf}$ are unique up to horizontal translations (see for instance \cites{GG,MG,GR} for similar uniqueness results).
\end{remark}

\section{Appendix}\label{appendix}
We prove in this appendix the assertions in Remark \ref{ossne}.
\begin{lemma}\label{unicoptocritico}
If $B_1$ is the unit ball of $\mathbb{R}^N$, then $0\in B_1$ is a non-degenerate global minimum of $V_{B_1}$ and it is the unique critical point.
\end{lemma}

\begin{proof}
First of all we note that $V_{B_1}$ is a radial function, i.e.\ $V_{B_1}(x)=v_{B_1}(|x|)$. Hence, since $V_{B_1}$ is smooth 
in the interior of the ball, it follows that $v_{B_1}'(0)=0$. It is easily seen that 
$$
  (\Delta V_{B_1})(0) =  2 (1+s) (N+2s) \int_{B_1^C}\frac{1}{|y|^{N+2s+2}}\di y > 0. 
$$
Therefore, since $v_{B_1}''(0)=\frac{1}{n}\Delta V_{B_1}(0)$, it follows that for 
 fixed $\delta>0$ one has $v_{B_1}''(t) > 0$ for $t\in [0,\delta]$, which implies the non-degeneracy of the 
 origin as a critical point of $V_{B_1}$. 

It remains to show the monotonicity of $v_{B_1}$ in the whole interval $(0,1)$, but since Lemma~\ref{comportamentoVomega} holds, it is sufficient to show that
\begin{equation}\label{crescente}
\frac{\di}{\di t}V_{B_1}(t \vec{e}_1)\neq 0 \quad \text{for}\; t\in [\delta, 1-\delta].
\end{equation}
Recalling the definition \eqref{Vomega}, we get
\begin{equation}\label{monotonia}
\frac{\di}{\di t}V_{B_1}(t \vec{e}_1)=\tilde{c}_{N,s}\int_{B_1^C}\frac{y_1-t}{|y-t\vec{e}_1|^{N+2s+2}}\di y,
\end{equation}
where $\tilde{c}_{N,s}$ is a constant depending only on $N$ and 
$s$, $y=(y_1,y')\in \mathbb{R}\times \mathbb{R}^{N-1}$ and $B_1^C$ denotes the complement of $B_1$.

By Fubini's Theorem 
\begin{equation}\label{fubini}
\int_{B_1^C}\frac{y_1-t}{|y-t\vec{e}_1|^{N+2s+2}}\di y=\int_{\mathbb{R}^{N-1}}\di y' \int_{\{y_1:(y_1,y')\in B_1^C\}}\frac{y_1-t}{|y-t\vec{e}_1|^{N+2s+2}}\di y.
\end{equation}
Since $(y_1,y')\in B_1^c \times \mathbb{R}^{N-1}$, we have two cases:
\begin{itemize}
\item[$1)$] if $|y'|\le 1 \quad \Rightarrow \quad y_1\in \mathbb{R}$;
\item[$2)$] if $|y'|<1 \quad \Rightarrow \quad y_1\le -\sqrt{1-|y'|^2}\; \vee \; y_1\ge \sqrt{1-|y'|^2}$.
\end{itemize}
In the first case we obtain by oddness 
\begin{equation}\label{monot1}
\int_{\{y_1:(y_1,y')\in B_1^C\}}\frac{y_1-t}{|y-t\vec{e}_1|^{N+2s+2}}\di y=\int_{\{y_1\in \mathbb{R}\}}\frac{y_1-t}{((y_1-t)^2+|y'|^2)^{(N+2s+2)/2}}\di y=0.
\end{equation}
In the second case, using the changes of variables $y_1-t=s$ and $z=t-y_1$, we get
\begin{equation}\label{monot}
\begin{aligned}
&\int_{\{y_1:(y_1,y')\in B_1^C\}}\frac{y_1-t}{|y-t\vec{e}_1|^{N+2s+2}}\di y\\
&=\int_{\{y_1\le -\sqrt{1-|y'|^2}\}}\frac{y_1-t}{|y-t\vec{e}_1|^{N+2s+2}}\di y+\int_{\{y_1\ge \sqrt{1-|y'|^2}\}}\frac{y_1-t}{|y-t\vec{e}_1|^{N+2s+2}}\di y\\
&=\int_{\{z\ge t+\sqrt{1-|y'|^2}\}}\frac{z}{(z^2+|y'|^2)^{(N+2s+2)/2}}\di z\\
&+\int_{\{s\ge \sqrt{1-|y'|^2}-t\}}\frac{s}{(s^2+|y'|^2)^{(N+2s+2)/2}}\di y>0,
\end{aligned}
\end{equation}
since $\{z:z\ge t+\sqrt{1-|y'|^2}\}\subseteq \{ z:z\ge \sqrt{1-|y'|^2}-t\}$ and since the first integral is negative.

Putting together $\eqref{monotonia}$, $\eqref{fubini}$, $\eqref{monot1}$ and $\eqref{monot}$ we obtain $\eqref{crescente}$, which concludes the proof.  \end{proof}

\medskip 

\begin{lemma}\label{l:5.2}
 Let $\Phi_{\xi}$ be defined as in Proposition \ref{criticoimplicaNMC}. There exist dumb-bell domains (as in Figure \ref{neck}) with the same topology of the ball such that $\Phi_\xi$ has at least three critical points.
\end{lemma}

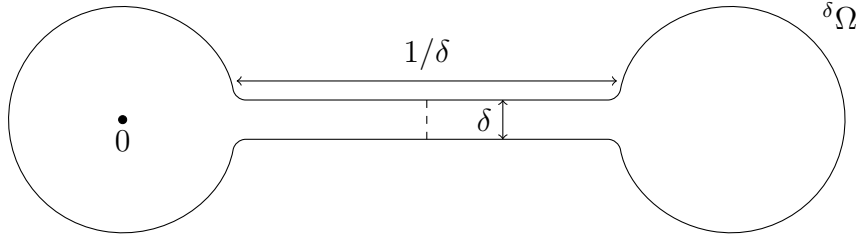
\begin{figure}[t]
  \centering
  \begin{tikzpicture}
    \pgfmathsetmacro{\r}{1.5};
    \pgfmathsetmacro{\a}{10};
    \pgfmathsetmacro{\x}{8};
    \draw[fill] (0,0) circle[radius=1.5pt] node[below] {$0$};
    \draw[rounded corners] (\x,0) +(+180-\a:\r cm) arc (+180-\a:-180+\a:\r cm) -- (-\a:\r cm) arc (360-\a:+\a:\r cm) -- cycle;
    \draw[dashed] (\x/2,0 |- +\a:\r) -- (\x/2,0 |- -\a:\r);
    \draw[<->] (\x/2+1,0 |- +\a:\r) -- (\x/2+1,0 |- -\a:\r)node[midway,left] {$\delta$};
    \draw[<->] (\r,0 |- +2*\a:\r) -- (\x-\r,0 |- +2*\a:\r)node[midway,above] {$1/\delta$};
    \draw (\x,0) +(45:\r cm) node[above right] {$^{\delta}\Omega$};
  \end{tikzpicture}
  \caption{A dumb-bell domain $^{\delta}\Omega$. }\label{neck}
\end{figure}
\end{remark}

\begin{proof}[Sketch of the Proof]
We consider a sequence of domains ${}^{\delta}\Omega$ as in Figure \ref{neck}. Fixed $r\in (0,1)$, it is easy to see that 
\begin{equation}
V_{{}^{\delta}\Omega}\rightarrow V_{B_1} \quad \text{in}\; C^2(B_r(0))\quad \text{as}\; \delta \rightarrow 0.
\end{equation}
For $\delta$ small, by Lemma \ref{unicoptocritico}, we get that $V_{{}^{\delta}\Omega}$ has a unique non-degenerate minimum $x_1$ in $B_{r/2}(0)$ and there exists $\gamma >0$ such that
\[
\inf_{\partial B_r(0)}V_{{}^{\delta}\Omega}>\sup_{B_{r/2}(0)}V_{{}^{\delta}\Omega}+\gamma.
\]
By symmetry, we have a non-degenerate minimum point $x_2$ in the other ball with the same properties. Recall also that from Lemma~\ref{comportamentoVomega} that if $x\in \partial {}^{\delta}\Omega$, it holds
\[
\lim_{{}^{\delta}\Omega \ni y\rightarrow x}V_{{}^{\delta}\Omega}(y)=+\infty.
\]
Hence, from \eqref{eq:exp-Phi-V} (with a similar formula for the gradient in $\xi$) and the 
above observations, there exists a critical point of $\Phi$ other that $x_1$ and $x_2$, 
by Mountain Pass Theorem. 
\end{proof}

\medskip We notice that the argument in the proof of Lemma \ref{l:5.2} is rather flexible, and 
does not require rigidity assumptions on the domain such as some symmetry.

\end{document}